\documentclass[psamsfonts]{amsart}
\usepackage[utf8]{inputenc}
\usepackage{textcomp}

\usepackage{amssymb,amsfonts}
\usepackage[all,arc]{xy}
\usepackage{enumerate}
\usepackage{mathrsfs}
\usepackage{hyperref}
\usepackage{xcolor}
\usepackage{graphicx}

\newtheorem{thm}{Theorem}[section]

\newtheorem{prop}[thm]{Proposition}
\newtheorem{lem}[thm]{Lemma}

\theoremstyle{definition}

\theoremstyle{remark}
\newtheorem{rem}[thm]{Remark}

\numberwithin{equation}{section}

\newcommand{\Z}{\mathbb{Z}}
\newcommand{\R}{\mathbb{R}}

\newcommand{\C}{\mathbb{C}}

\title{A Proof of the Area Vanishing Square Peg Conjecture}

\author{Lorenzo Valerio}
\address{Department of Mathematics,
University of Maryland,
College Park, MD 20742, USA}
\email{lorenz13@umd.edu}

\begin{document}

\begin{abstract}

We resolve an open problem posed by Tao related to the Toeplitz square peg conjecture. We prove that four simple closed polygonal curves in the cylinder, each winding once in the positive direction, jointly inscribe a square whenever the alternating sum of their $\int y\, dx$ integrals vanishes. The proof reduces the problem to the non-displaceability of a pair of Lagrangian tori, which we establish using Lagrangian Floer homology.
\end{abstract}

\maketitle

\section{Introduction}

It is a longstanding open problem of Toeplitz \cite{Toeplitz} whether every Jordan curve in the plane contains four points corresponding to the vertices of a square. In recent years, much progress has been made for curves with certain regularity assumptions. In particular, symplectic methods have been useful in solving many variants of the problem, including the smooth rectangular peg problem, solved by Greene and Lobb in \cite{Greene Lobb rectangle}. 

However, in \cite{Tao integration approach}, Tao argues that arguments that assume certain regularity conditions for the curve don't seem to translate well to the case of arbitrarily rough curves. For this reason, Tao has introduced an integration approach to the square peg problem. In his paper, he poses a new series of conjectures relating to the original square peg problem, providing a possible road map to solving it. The following is Conjecture 4.1 in \cite{Tao integration approach}.

\begin{thm}[\textup{Periodic square peg problem}]
    Let $\sigma_1, \sigma_2 \colon \mathbb{R}/\mathbb{Z} \to Cyl$ be simple curves in the cylinder $Cyl=(\R/\Z)\times \R$  homologous to the circle $\R/\Z$. Suppose also that the sets $\sigma_1(\mathbb{R}/\mathbb{Z})$ and $\sigma_2(\mathbb{R}/\mathbb{Z})$ are disjoint. Then $\sigma_1(\mathbb{R}/\mathbb{Z}) \cup \sigma_2(\mathbb{R}/\mathbb{Z})$ inscribes a square.
\end{thm}

This theorem has been proven by Hugelmeyer \cite{Hugelmeyer} using Lagrangian Floer homology. In this paper, we extend his approach to the setting of Conjecture 5.6 in \cite{Tao integration approach}. We thereby present the proof of the following theorem, which is stated as the contrapositive of Conjecture 5.6 in \cite{Tao integration approach}.

\begin{thm}[\textup Area Vanishing]\label{thm: area inequality}
    Let $\sigma_1,\sigma_2,\sigma_3, \sigma_4: \R/\Z\to Cyl$ be simple closed polygonal paths isotopic to $\mathrm{Graph}_{0,1}:\R/\Z\to Cyl$ where $\mathrm{Graph}_{0,1}(t):=(t,0)$. If \begin{equation}\label{Tao area}
        \int_{\sigma_1}y\ dx - \int_{\sigma_2}y\ dx+\int_{\sigma_3}y\ dx-\int_{\sigma_4}y\ dx= 0
    \end{equation} then $$(\sigma_1(\R/\Z), \sigma_2(\R/\Z), \sigma_3(\R/\Z), \sigma_4(\R/\Z))$$ jointly inscribe a square.
\end{thm}

For Tao, four sets $\Gamma_1, \Gamma_2, \Gamma_3, \Gamma_4$ in $Cyl$, possibly overlapping, are said to \textit{jointly inscribe a square} if there exist $x\in \R/\Z$ and $y,a,b\in \R$ such that \[ \begin{aligned} (x,y) &\in \Gamma_1,\\ (x+a,y+b) &\in \Gamma_2,\\ (x+a-b,y+a+b) &\in \Gamma_3,\\ (x-b,y+a) &\in \Gamma_4. \end{aligned} \] Note that degenerate squares are permitted in this definition, and that the order of the $\Gamma_i$ matters.

Naseri Sadr \cite{Ali} also recently proved Tao's Conjecture 5.6 as a special case of a more general inscription theorem for isosceles trapezoids. His proof uses Lagrangian Floer homology for noncompact Lagrangian cylinders in $\R\times S^1\times\C$. The proof presented here instead works with Lagrangian tori in a finite symplectic cover of $T^4=(\C/\Z[i])^2$, as we will see.

The statement of Theorem \ref{thm: area inequality} is the strongest of the conjectures posed in \cite{Tao integration approach} as it implies all the other ones in that paper except for the original square peg conjecture. We will prove that it is a consequence of the following result, which serves a similar purpose to Theorem 2 in \cite{Hugelmeyer}. The distinction is that of working with four curves instead of two, as well as having the alternating area condition, which will play a key role in the Floer computation later on.

\begin{thm}\label{thm: area inequality torus}
    Let $\sigma_i:S^1\to \C/\Z[i]$ be smooth embeddings isotopic to the circle $\R/\Z$ for $i=1,2,3,4$, and suppose that
    \begin{equation}\label{eqn:area vanishing torus}
        A(\sigma_1)-A(\sigma_2)+A(\sigma_3)-A(\sigma_4)=0.
    \end{equation}
    Then there exist $a_1,a_2,a_3,a_4\in S^1$ such that $\sigma_4(a_4)=\sigma_1(a_1)+i(\sigma_2(a_2)-\sigma_1(a_1))$ and $\sigma_3(a_3)=\sigma_2(a_2)+i(\sigma_2(a_2)-\sigma_1(a_1))$.
\end{thm}

\begin{figure} \centering \includegraphics[width=0.5\textwidth]{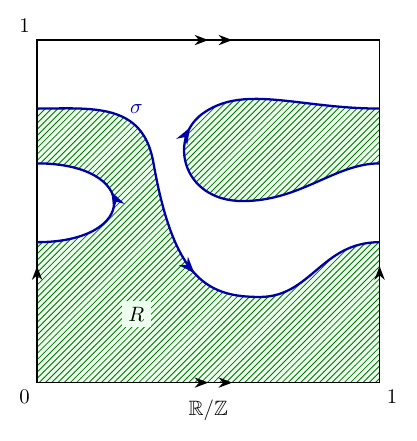}\caption{The shaded region is $A(\sigma)=\operatorname{Area}(R)\in\mathbb R/\mathbb Z$} \label{figure1}\end{figure}

Here, $A(\sigma)$ is the integral $\int_{\sigma}y\ dx$ adapted to the torus $\C/\Z[i]$. More specifically, if $\sigma:S^1\to \C/\Z[i]$ is a smooth embedding isotopic to the circle $\R/\Z$, let $\hat{\sigma}:\R\to \C$ be a lift of the map $t\mapsto \sigma(t+\Z)$ such that $\hat{\sigma}(t)=x(t)+i \ y(t)$ and $\hat{\sigma}(t+1)=\hat{\sigma}(t)+1$ and define \begin{equation}
    A(\sigma)=\int_0^1y(t)\ dx(t)\ \mod \Z
\end{equation} 
This integral can be thought of as the area ``enclosed'' by the curve $\hat\sigma$ and the $x$-axis in a unit time interval. Figure \ref{figure1} represents $A(\sigma)$ in the torus. Note that it is well-defined: a different lift $\hat\sigma'$ differs by $m+i\,n\in \Z[i]$, so the integral will differ by $n\in \Z$, and hence is well-defined in $\R/\Z$.

\subsection{On AI use} There was minimal use of AI tools in the production of this paper. I have used Gemini 3.1 Pro mainly as a learning tool and as a reference provider. In particular, it pointed me towards Denis Auroux's notes to learn about Lagrangian Floer homology \cite{Denis Aroux}. I have also used Gemini 3.1 Pro and GPT-6 Astra to identify errors and clarify parts of the arguments in the final drafts of the paper. The TikZ code for the figures has been generated by GPT-6 Astra.

\section{Setup and Proofs of Main Theorems}

The setup for the proof of Theorem \ref{thm: area inequality torus} is the same as the one in \cite{Hugelmeyer}, but we set everything up here for completeness. Let $\omega=dx\wedge dy = \frac{i}{2}dz\wedge d\bar{z}$ be the symplectic area form on $\C/\Z[i]$ induced by the standard symplectic form on $\C$. Denote the four-torus $(\C/\Z[i])^2$ by $T^4$, and let $\pi_1, \pi_2:T^4\to \C/\Z[i]$ be the projection maps onto the first and second factor, respectively. Define the symplectic form $\omega_\pm=\pi_1^*\omega - \pi_2^*\omega$ on $T^4$. The ``make a square map'' (Figure \ref{figure2}) $\tau:(T^4, \omega_{\pm})\to (T^4, \omega_{\pm})$ is defined in \cite{Hugelmeyer} as $\tau(a,b)=(a+i(b-a), b+i(b-a))$. It is clearly smooth and has smooth inverse $\tau^{-1}(a,b)=(a-i(b-a), b-i(b-a))$, and a quick computation reveals $\tau^*\omega_\pm=\omega_\pm$, so $\tau$ is a symplectomorphism.

\begin{figure} \centering \includegraphics[width=0.5\textwidth]{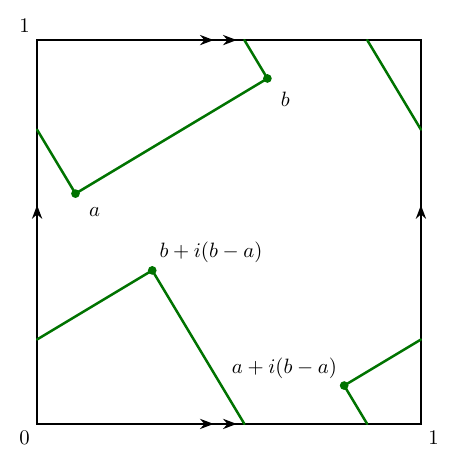}\caption{The \textit{make-a-square map} $\tau$} \label{figure2}\end{figure}

 Consider the Lagrangian submanifolds $L_0=(\sigma_1\times \sigma_2)(S^1\times S^1)\subset T^4$ and $L_1=(\sigma_4\times \sigma_3)(S^1\times S^1)\subset T^4$, where the $\sigma_i$ are as in the statement of Theorem \ref{thm: area inequality torus}. Proving Theorem \ref{thm: area inequality torus} is equivalent to proving that $\tau(L_0)\cap L_1\neq \varnothing$, because then there exist $a_1,a_2,a_3,a_4\in S^1$ such that $\sigma_4(a_4)=\sigma_1(a_1)+i(\sigma_2(a_2)-\sigma_1(a_1))$ and $\sigma_3(a_3)=\sigma_2(a_2)+i(\sigma_2(a_2)-\sigma_1(a_1))$, that is, the $\sigma_i$ jointly inscribe a square in $T^2$. 
 
 In fact, we can simplify the problem further by considering the straight curves  \[\Gamma_j =\{\ x+iA(\sigma_j)\ |\ x\in \R/\Z\ \}\subset \C/\Z[i].\] These curves are Hamiltonian isotopic to the $\sigma_j$ for each $j=1,2,3,4$ respectively (see Lemma \ref{lem: Hamiltonian isotopy}.) So now we can define the ``straightened-out'' Lagrangian submanifolds $K_0=\Gamma_1\times \Gamma_2\subset T^4$ and $K_1=\Gamma_4\times \Gamma_3\subset T^4$ which are Hamiltonian isotopic to $L_0$ and $L_1$ respectively.  Moreover, since $\tau$ is a symplectomorphism, $\tau(L_0)$ is Hamiltonian isotopic to $\tau(K_0)$. Recall that two submanifolds of a symplectic manifold $N_1, N_2\subset M$ are \textit{Hamiltonian non-displaceable}, or simply \textit{non-displaceable}, if every Hamiltonian diffeomorphism $\phi:M\to M$ satisfies $N_1\cap \phi(N_2)\neq \varnothing$. Therefore, to show that $\tau(L_0)\cap L_1\neq \varnothing$, and hence Theorem \ref{thm: area inequality torus}, it suffices to show the following proposition.

 \begin{prop}
     Let $K_0,K_1\subset T^4$ be the Lagrangian submanifolds defined as above under the hypotheses of Theorem \ref{thm: area inequality torus}. Then $\tau(K_0)$ and $K_1$ are Hamiltonian non-displaceable.
 \end{prop}

\begin{proof}
    Let $p:(T^4, p^*\omega_{\pm})\to (T^4, \omega_{\pm})$ be the 4-fold symplectic covering map defined as $p(z_1, z_2)=(z_1+z_2, \overline{z_1-z_2})$. Note that $p^*\omega_{\pm} = p^*(\frac{i}{2}(dz_1\wedge d\bar{z_1}-dz_2\wedge d\bar{z_2}))=i(dz_1\wedge d\bar{z_1} + dz_2\wedge d\bar{z_2})=2\pi_1^*\omega +2\pi_2^*\omega$, so our covering map actually splits like so $p:((\C/\Z[i])\times (\C/\Z[i]), 2\omega\oplus 2\omega)\to ((\C/\Z[i])^2, \omega_{\pm})$. 
    
    Pick real representatives $c_1,c_2,c_3,c_4\in \R$ of the  $A(\sigma_1), A(\sigma_2), A(\sigma_3), A(\sigma_4)$ so that $c_1-c_2+c_3-c_4=0$ holds (we can do this since we are assuming \ref{eqn:area vanishing torus} holds), and let     
    \begin{align}
        d_{0,A} &= \frac{c_1-c_2}{2} \mod \Z, &
        d_{0,B} &= \frac{c_1+c_2}{2} \mod \Z, \nonumber\\
        d_{1,A} &= \frac{c_4-c_3}{2}\mod \Z, &
        d_{1,B} &= \frac{c_4+c_3}{2}\mod \Z
    \end{align}
    
    Now we can define $\gamma_{0,A},\gamma_{0,B},\gamma_{1,A},\gamma_{1,B}:S^1\to \C/\Z[i]$ by 
    \begin{align}
        \gamma_{0,A}(t) &= t + id_{0,A}, &
        \gamma_{0,B}(t) &= (1-2i)t + id_{0,B}, \nonumber\\
        \gamma_{1,A}(t) &= t + id_{1,A}, &
        \gamma_{1,B}(t) &= t + id_{1,B}
    \end{align}
    
    Moreover, define the tori $\tilde{K_1}=\gamma_{1,A}\times \gamma_{1,B}$ and $\widetilde{\tau(K_0)}=\gamma_{0,A}\times \gamma_{0,B}$, which are connected components of $p^{-1}(K_1)$ and $p^{-1}(\tau(K_0))$ respectively. In fact, one can check via computation that the restrictions 
    \[p|_{\tilde{K}_1}:\tilde{K_1}\to K_1\ \text{ and }\ p|_{\widetilde{\tau(K_0)}}:\widetilde{\tau(K_0)}\to \tau(K_0)\] are two-sheeted covering maps. Since a Hamiltonian isotopy in the base space induces a Hamiltonian isotopy in the covering space, it suffices to show that $\widetilde{\tau(K_0)}, \tilde{K_1}$ are non-displaceable, and Lemma \ref{lem: Hamiltonian isotopy covering} implies that $\tau(K_0), K_1$ are non-displaceable.

    To prove $\tilde{K}_1,\widetilde{\tau(K_0)}$ are non-displaceable, we will use Lagrangian Floer homology to show that $HF(\tilde{K_1},\widetilde{\tau(K_0)};q;\Z_2)\neq 0$. This is the Lagrangian Floer homology restricted to the connected component of the path space containing the constant path at the intersection point $q$ (see \cite{HLS2016} Section 2.1).  By showing that this component of the homology is nonzero, if a Hamiltonian diffeomorphism displaced the pair $(\tilde{K_1}, \widetilde{\tau(K_0)})$, the corresponding Floer complex would have no generators, and hence zero homology, contradicting the invariance of Floer homology under Hamiltonian isotopies. 
    
    First of all, in our case, Floer homology is unobstructed since $\pi_2(T^4)=0$ and $\pi_2(T^4,\tilde{K_1})=\pi_2(T^4,\widetilde{\tau(K_0}))=0$  \cite{FOOO} (the latter follows from the long exact sequence of relative homotopy groups and the observation that the maps $\pi_1(\tilde{K_1})\to \pi_1(T^4)$ and $\pi_1(\widetilde{\tau(K_0)})\to \pi_1(T^4)$ are injective). 
    
    Now we apply the Künneth-type formula for $HF$ derived in Section 2.6 of \cite{HLS2016} (see Remark \ref{remark: kunneth}.) For this we pick a point $a\in \gamma_{1,A}\cap \gamma_{0,A}$ and $b\in \gamma_{1,B}\cap \gamma_{0,B}$, and let $q=(a,b)$. We obtain
    \begin{align*}
        HF(\tilde{K_1}, \widetilde{\tau(K_0)};q; \Z_2)&=HF(\gamma_{1,A}\times \gamma_{1,B}, \gamma_{0,A}\times \gamma_{0,B};q; \Z_2)\\
        &\cong HF(\gamma_{1,A}, \gamma_{0,A};a; \Z_2)\otimes_{\Z_2} HF(\gamma_{1,B}, \gamma_{0,B};b;\Z_2)
    \end{align*}
    
    Since we are assuming equation \eqref{eqn:area vanishing torus} holds, we have $c_1-c_2=c_4-c_3$ and so $\gamma_{0,A}=\gamma_{1,A}$. Thus, $HF(\gamma_{1,A}, \gamma_{0,A};a;\Z_2)=HF(\gamma_{1,A}, \gamma_{1,A};a;\Z_2)\cong H^*(S^1;\Z_2)\cong \Z_2\oplus \Z_2$ (see Section 2.1.1 in \cite{HLS2016} or Proposition 1.13 in \cite{Denis Aroux}.) 
    
    Moreover, notice that the curves $\gamma_{0,B}$ and $\gamma_{1,B}$ intersect transversely exactly twice with the same local sign (see Figure \ref{figure3}). The differential $\partial$ of the Floer complex counts the pseudo-holomorphic strips of $\operatorname{ind}(u)=1$ joining intersection points. We claim that $\partial=0$ and hence $HF( \gamma_{1,B},\gamma_{0,B};b;\Z_2)\cong CF(\gamma_{1,B},\gamma_{0,B};b)\neq 0$. Since $\gamma_{0,B}, \gamma_{1,B}$ are oriented Lagrangians, we have a $\Z_2$-grading on the Floer complex (see \cite{Denis Aroux} Section 1.3). In particular, if $\gamma_{1,B}\cap \gamma_{0,B}=\{b,b'\}$ then the local intersection sign at $b$ determines the $\mathbb Z_2$-degree $\deg(b)$, and similarly for $\deg(b')$. Because $\gamma_{0,B}$ and $\gamma_{1,B}$ intersect with the same sign at $b$ and at $b'$, we have $\deg(b)=\deg(b')$ in the $\Z_2$-grading. Thus $\operatorname{ind}(u)\equiv\deg(b)-\deg(b')\equiv0 \pmod 2$ for any strip $u$, which means $\partial=0$.

    \begin{figure} \centering \includegraphics[width=0.5\textwidth]{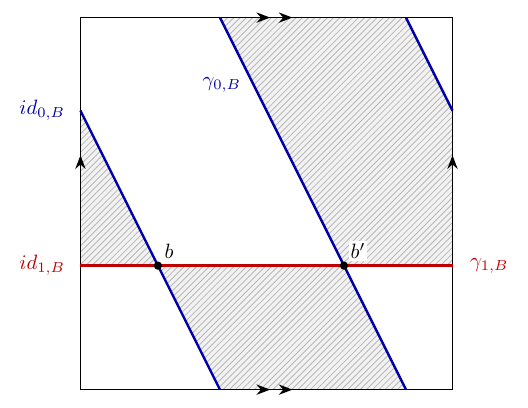}\caption{$\gamma_{0,B}$ and $\gamma_{1,B}$ intersect with the same sign at $b$ and $b'$.}\label{figure3}\end{figure}
    
    Therefore, we obtain the following $$\dim_{\Z_2} HF(\tilde{K_1}, \widetilde{\tau(K_0)};q;\Z_2)=\dim_{\Z_2}  HF(\gamma_{1,A}, \gamma_{0,A};a;\Z_2)\cdot \dim_{\Z_2} HF(\gamma_{1,B}, \gamma_{0,B};b;\Z_2)\geq 2$$     
    and conclude that $HF(\tilde{K_1}, \widetilde{\tau(K_0)};q;\Z_2)\neq 0$ which completes the proof.    
\end{proof}

This also concludes the proof of Theorem \ref{thm: area inequality torus}. 

\begin{rem}\label{remark: kunneth}
    To apply the Künneth-type formula in \cite{HLS2016}, we need to verify two things. One is that $(\C/\Z[i], 2\omega)$ is symplectically aspherical, i.e. that $(2\omega)|_{\pi_2(\C/\Z[i])}=0$. But this is true since $\pi_2(\C/\Z[i])=0$. The second is that the pair of Lagrangians $\gamma_{1,A}, \gamma_{0,A}\subset \C/\Z[i]$ must be \textit{weakly exact} with respect to any intersection point $p\in \gamma_{1,A}\cap \gamma_{0,A}$ (see Section 2.1 in \cite{HLS2016}), and the same needs to hold for $\gamma_{1,B},\gamma_{0,B}\subset \C/\Z[i]$. In particular, this means checking that $\int_{D}u^*(2\omega)=0$ for any map $u:(D, \partial D^+, \partial D^-, \{+1,-1\})\to (\C/\Z[i], \gamma_{1,\star}, \gamma_{0,\star}, \{p\})$ where $D$ is the unit disk. Let $\hat{u}:D\to \R^2$ be a lift of the map $u$ (we can lift it since $D$ is simply connected). Let $\hat p=\hat{u}(1)\in \R^2$ and define the 1-form $\lambda = (x-\hat p_x)\,dy - (y-\hat{p}_y)\, dx$. Note that $d\lambda = 2dx\wedge dy=q^*(2\omega)$ where $q:\C\to \C/\Z[i]$ is the quotient map. Moreover, note that $\lambda|_{l}=0$ for any line $l\subset \R^2$ passing through $\hat p$. In particular, $\lambda$ vanishes on the lines through $\hat p$ which are in $q^{-1}(\gamma_{0,\star})$ and $q^{-1}(\gamma_{1,\star})$. So, by Stokes' theorem, we obtain \begin{align*}
        \int_{D}u^*(2\omega) &= \int_D \hat u^*(q^*(2\omega))=\int_D \hat u ^*d\lambda = \int_{\partial D^+}\hat u^*\lambda + \int_{\partial D^-}\hat u^* \lambda =0.
    \end{align*}
    Moreover, as discussed in \cite{HLS2016}, having checked these conditions also allows us to consider Floer homology with $\Z_2$-coefficients.
\end{rem}

Now it remains to show that Theorem \ref{thm: area inequality torus} implies Theorem \ref{thm: area inequality}.

\begin{proof}[Proof of Theorem~\ref{thm: area inequality}]
    Let $\sigma_{i,n}$ be smooth embeddings $S^1\to Cyl=(\R/\Z)\times \R$ isotopic to the circle $\R/\Z$ for $i=1,2,3,4$ such that they converge in the $C^0$-topology to the maps $\sigma_i$ in Theorem \ref{thm: area inequality}. Furthermore, we choose the approximations $\sigma_{i,n}$ so that
    \begin{equation}\label{area vanishing}
        \int_{\sigma_{1,n}}y\ dx - \int_{\sigma_{2,n}}y\ dx + \int_{\sigma_{3,n}}y\ dx - \int_{\sigma_{4,n}}y\ dx =0
    \end{equation}
    for all $n$. We can do this by picking smooth embedded approximations $\hat\sigma_{i,n}$ of the $\sigma_i$ converging uniformly and in the $C^1$-topology on compact subsets away from the corner parameters of $\sigma_i$, with uniformly bounded first derivatives. Then we have
    \[
    \int_{\hat\sigma_{1,n}}y\,dx-\int_{\hat\sigma_{2,n}}y\,dx+\int_{\hat\sigma_{3,n}}y\,dx-\int_{\hat\sigma_{4,n}}y\,dx    =\varepsilon_n  \]
    where $\varepsilon_n\to 0$ as $n\to\infty$ by the dominated convergence theorem and the area-vanishing hypothesis on the $\sigma_i$. Now, for each $n$, define $\sigma_{1,n}(t)=\hat\sigma_{1,n}(t)-(0,\varepsilon_n)$ so we get $\int_{\sigma_{1,n}}y\ dx= \int_{\hat\sigma_{1,n}}y\ dx-\varepsilon_n$. Since $\varepsilon_n\to 0$ then $\sigma_{1,n}\to \sigma_1$. Set ${\sigma}_{i,n}=\hat\sigma_{i,n}$ for $i=2,3,4$ and obtain \eqref{area vanishing}.
    
    Now let $\lambda>0$ such that $im(\sigma_{i,n})\subset S^1\times [-\lambda, \lambda]$ for all $i,n$, and let $N\in \Z$ such that $N>16\lambda$. For each $i,n$, choose a lift $\overline\sigma_{i,n}:\R\to \C$ of the $\sigma_{i,n}$ so that $\overline\sigma_{i,n}(t+1)=\overline\sigma_{i,n}(t)+1$. Now define $\tilde{\sigma}_{i,n}(t) = \frac{1}{N}\overline\sigma_{i,n}(Nt)+\Z[i]$. Note that the $\tilde{\sigma}_{i,n}$ are well-defined as smooth embeddings $S^1\to \C/\Z[i]$. Moreover, we have $A(\tilde{\sigma}_{1,n})-A(\tilde{\sigma}_{2,n})
    +A(\tilde{\sigma}_{3,n})-A(\tilde{\sigma}_{4,n})
    =0$. To see this, let $\overline\sigma_{i,n}(t) = x(t)+i\ y(t)$. We obtain: \begin{align*}
    A(\tilde{\sigma}_{i,n})
    &\equiv
    \int_0^1\frac{1}{N} y(Nt)\,
        d\left(\frac{1}{N} x(Nt)\right)\\
    &=\frac{1}{N}\int_0^1 y(Nt)x'(Nt)\,dt\\
    &=\frac{1}{N^2}\int_0^N y(s)x'(s)\,ds\\
    &=\frac{1}{N}\int_{\sigma_{i,n}}y\,dx
    \pmod{\mathbb Z},
        \end{align*}
    where the last equality uses the $1$-periodicity of
    $y(s)x'(s)$ and the fact that $N$ is an integer.
    Consequently, equation \eqref{area vanishing} implies
    \[
    A(\tilde{\sigma}_{1,n})-A(\tilde{\sigma}_{2,n})
    +A(\tilde{\sigma}_{3,n})-A(\tilde{\sigma}_{4,n})
    =0 
    \]
    
    Therefore we can apply Theorem \ref{thm: area inequality torus} to our $\tilde{\sigma}_{i,n}$ to obtain the points $a_1, a_2, a_3, a_4\in S^1$ that form the square, with $\tilde{\sigma}_{4,n}(a_4)=\tilde{\sigma}_{1,n}(a_1)+i(\tilde{\sigma}_{2,n}(a_2)-\tilde{\sigma}_{1,n}(a_1))$ and $\tilde{\sigma}_{3,n}(a_3)=\tilde{\sigma}_{2,n}(a_2)+i(\tilde{\sigma}_{2,n}(a_2)-\tilde{\sigma}_{1,n}(a_1))$ in the torus $\C/\Z[i]$. Now pick representatives of these points in $\R^2\cong \C$ lying inside the strip $\R\times [-\frac{1}{16}, \frac{1}{16}]$ in such a way that they form a square with diagonal length less than $1/\sqrt{32}$. We can pick corresponding points in the $\text{im}(\sigma_{i,n})$ which will now form a square with diagonal length less than $N/\sqrt{32}$. The square lives within the region $(\R/\Z)\times [-\lambda, \lambda]\subset Cyl$, so applying compactness, we can find a sequence of squares jointly inscribed by the $\sigma_{i,n}$ which converges to a square jointly inscribed by the $\sigma_i$.
\end{proof}

\section{Lemmas}

In this section we prove the lemmas used in the main proofs.

\begin{figure} \centering \includegraphics[width=0.5\textwidth]{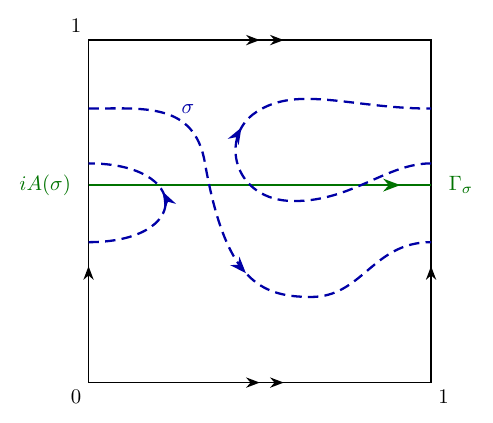}\caption{The curves $\sigma$ and $\Gamma$ from Lemma \ref{lem: Hamiltonian isotopy} are Hamiltonian isotopic.} \label{figure4}\end{figure}

\begin{lem}\label{lem: Hamiltonian isotopy}
    A smooth embedding $\sigma:S^1\to \C/\Z[i]$ isotopic to the circle $\R/\Z$ is Hamiltonian isotopic to the curve $\Gamma=\{\ x+iA(\sigma)\ |\ x\in \R/\Z\ \}\subset \C/\Z[i]$.
\end{lem}

\begin{proof}
    Let $\sigma_s:S^1\to\C/\Z[i]$ be an isotopy of curves so that $\sigma_0=\sigma$ and $\sigma_1(S^1)=\R/\Z$. Let $\gamma_s(t)=\sigma_s(t)+i\cdot c(s)$ be another family of circles on $\C/\Z[i]$, where $c(s)=A(\sigma)-A(\sigma_s)$ is a vertical shift so that $A(\gamma_s)=A(\sigma)$ for all $s\in [0,1]$, so we get an isotopy from $\gamma_0(S^1)=\sigma(S^1)$ to $\gamma_1(S^1)=\Gamma$. Flowing along the vector field $\partial_s\gamma_s(t)$ gives this isotopy. We want this vector field to be the evaluation of a Hamiltonian vector field coming from a map $H_s:\C/\Z[i]\to \R$, i.e., we want $X_{H_s}(\gamma_s(t))=\partial_s\gamma_s(t)$ where $X_{H_s}$ satisfies $\iota_{X_{H_s}}\omega = dH_s$ and $\omega$ is the symplectic area form on $\C/\Z[i]$. 
    
    Consider the 1-form $\alpha_s=\omega(\partial_s\gamma_s, \partial_t\gamma_s)\, dt$ on $S^1$. We will show this form is exact by showing $\int_{S^1}\alpha_s=0$. If we let $\hat{\gamma}_s:\R\to \C$ be a smooth family of lifts of the maps $t\mapsto\gamma_s(t+\Z)$, and write $\hat{\gamma}_s(t) = X(s,t)+i\,Y(s,t)$, then \[\alpha_s = (dx\wedge dy)(\partial_s\gamma_s, \partial_t\gamma_s)\, dt = (\partial_sX\partial_tY - \partial_tX\partial_sY)dt\] Now define $I(s)=\int_0^1 Y\partial_tX\ dt$ and note $A(\gamma_s)\equiv I(s)\mod \Z$. Since $A(\gamma_s)$ is constant and $I(s)$ is continuous, integrating by parts yields $$0=I'(s)=\int_0^1(\partial_tX\partial_sY + Y\partial_{ts}X)\, dt = [Y\partial_s X]_{t=0}^{t=1} - \int_{S^1}\alpha_s = -\int_{S^1}\alpha_s$$ since $Y(s,0)=Y(s,1)$ and $X(s,1)=X(s,0)+1$ so $\partial_sX(s,1)=\partial_sX(s,0)$. 

    Now we know $\alpha_s$ is exact, so write $\alpha_s=d\lambda_s$ where $\lambda_s:S^1\to \R$ is a smooth function with smooth dependence on $s$ and such that $\lambda_s(0)=0$. We want to find a Hamiltonian $H_s:\C/\Z[i]\to\R$ so that, in particular, \[\alpha_s=\omega(\partial_s\gamma_s, \partial_t\gamma_s)dt = dH_s(\partial_t\gamma_s)dt = d(H_s\circ \gamma_s),\] so we want $(H_s\circ \gamma_s)(t) = \lambda_s(t)$. Let $\Phi_s(t,r): S^1\times (-\varepsilon, \varepsilon)\to \C/\Z[i]$ be a tubular neighborhood of $\gamma_s$ smoothly depending on $(s,t,r)$ with uniform tubular radius $\varepsilon>0$ and so that $\Phi_s(t,0)=\gamma_s(t)$. Let $\phi:\R \to \R$ be a smooth bump function with $\phi(0)=1$ and $\phi(r)=0$ for $|r|\geq \varepsilon$. Now we can define $H_s=0$ on $(\C/\Z[i])\setminus \Phi_s(S^1\times (-\varepsilon, \varepsilon))$, and $H_s(\Phi_s(t,r))=\phi(r)\cdot \lambda_s(t)$ so that we get $H_s(\Phi_s(t,0))=H_s(\gamma_s(t))=\lambda_s(t)$. This construction yields a smooth Hamiltonian $H_s:\C/\Z[i]\to \R$.

    Now we can compute the Hamiltonian vector field $X_{H_s}$ along $\gamma_s$ \begin{align*}
        &\omega(X_{H_s}(\gamma_s(t)), \partial_t\gamma_s)dt=d(H_s\circ \gamma_s)=\alpha_s=\omega(\partial_s\gamma_s, \partial_t\gamma_s)dt\\
        \implies &\omega(X_{H_s}(\gamma_s(t))-\partial_s\gamma_s, \partial_t\gamma_s)=0
    \end{align*} 
    This means that \[X_{H_s}(\gamma_s(t))=\partial_s\gamma_s(t) +a(s,t)\partial_t\gamma_s(t)\] for some smooth function $a:[0,1]\times S^1\to \R$. To fix this tangential error term, we can reparametrize $\gamma_s$ to absorb the tangential component $a(s,t)\partial_t\gamma_s(t)$. Let $\tilde{\gamma}_s(t)=\gamma_s(\psi_s(t))$ where $\psi_s$ is the flow on $S^1$ of the time-dependent vector field $a(s,t)\partial_t$, with $\psi_0=\mathrm{id}$. The chain rule gives  \begin{equation}\label{ode}
    \partial_s\bigl(\tilde{\gamma}_s(t)\bigr) =X_{H_s}\bigl (\tilde\gamma_s(t)\bigr)\ \text{ with }\ \tilde\gamma_0(t)=\gamma_0(t).
    \end{equation}
    
    On the other hand, let $\varphi_s:\C/\Z[i]\to \C/\Z[i]$ be the flow of our Hamiltonian vector field $X_{H_s}$. Consider the path $s\mapsto \varphi_s(\gamma_0(t))$, and its flow equation at $\gamma_0(t)$ which is  
    \[\partial_s\varphi_s(\gamma_0(t))=X_{H_s}(\varphi_s(\gamma_0(t)))\ \text{ with }\ \varphi_0(\gamma_0(t))=\gamma_0(t)\] 
    
    This is the exact same ordinary differential equation and initial condition as \eqref{ode}. By Picard-Lindelöf, we get $\varphi_s(\gamma_0(t))=\tilde{\gamma}_s(t)$. This completes the proof that $\sigma$ is Hamiltonian isotopic to $\Gamma$ since $\varphi_0(\gamma_0(S^1))=\gamma_0(S^1)=\sigma(S^1)$ and $\varphi_1(\gamma_0(S^1))=\varphi_1(\sigma(S^1))=\tilde{\gamma}_1(S^1)=\gamma_1(S^1)=\Gamma$.  
 \end{proof}

 \begin{lem}\label{lem: Hamiltonian isotopy covering}
    Let $p:(\widetilde{X}, p^*\omega)\to (X,\omega)$ be a finite symplectic covering map, where $(X,\omega)$ is a symplectic manifold. Let $\widetilde{L}_j\subset \widetilde{X}$ be submanifolds in the cover so that $p(\widetilde{L}_j)=L_j$ for $j=0,1$ where the $L_j\subset X$ are Lagrangian submanifolds. If $L_0, L_1$ are displaceable, then $\widetilde{L}_0, \widetilde{L}_1$ are displaceable.
\end{lem}

\begin{proof}
   Suppose there is a Hamiltonian diffeomorphism $\varphi:X\to X$ such that $L_0\cap \varphi(L_1)=\varnothing$. We know $\varphi=\varphi_1$ is the time-1 flow of a Hamiltonian isotopy $\varphi_t$. By the homotopy lifting property, we can find $\tilde{\varphi_t}:\widetilde{X}\to \widetilde{X}$, the unique lift of $\varphi_t\circ p$ satisfying $\tilde{\varphi}_0=\text{id}_{\tilde{X}}$ and $\tilde{\varphi}=\tilde{\varphi}_1$. In fact, $\tilde{\varphi}$ is a Hamiltonian diffeomorphism: let $H_t:X\to \R$ be the time-dependent Hamiltonian that defines the flow $\varphi_t$ of the Hamiltonian vector field $X_{H_t}$ satisfying $\iota_{X_{H_t}}\omega = dH_t$. We can define the unique vector field $\widetilde{X}_{H_t}:\tilde{X}\to T\tilde{X}$ which is a lift of $X_{H_t}$ by $\widetilde{X}_{H_t}(\tilde{x})=(dp_{\tilde{x}})^{-1}(X_{H_t}(p(\tilde{x})))$. 
   
   Now we can differentiate both sides of the equation $p(\widetilde{\varphi}_t (\tilde{x})) = \varphi_t(p(\tilde{x}))$ to obtain
   
   \[dp_{\widetilde{\varphi}_t(\tilde{x})}\left(\frac{d}{dt}\widetilde{\varphi}_t(\tilde{x})\right) = \frac{d}{dt}\varphi_t(p(\tilde{x})) = X_{H_t}(\varphi_t(p(\tilde{x}))) = X_{H_t}(p(\widetilde{\varphi}_t(\tilde{x}))).\]

   Since $p$ is a local diffeomorphism we can apply the isomorphism $(dp_{\widetilde{\varphi}_t(\tilde{x})})^{-1}$ on both sides to obtain \[\frac{d}{dt}\widetilde{\varphi}_t(\tilde{x}) = \widetilde{X}_{H_t}(\widetilde{\varphi}_t(\tilde{x})) =  X_{H_t \circ p}(\widetilde{\varphi}_t(\tilde{x}))\]   where the last equality follows since    \[\iota_{\widetilde{X}_{H_t}}p^*\omega = p^*(\iota_{X_{H_t}}\omega)=p^*(dH_t)=d(H_t\circ p)\]

    So $\widetilde{\varphi}_t$ is the flow of the Hamiltonian vector field $X_{H_t\circ p}$, and thus $\widetilde{\varphi}$ is a Hamiltonian diffeomorphism.
    
    Now note that  $p(\tilde{L}_0\cap \tilde{\varphi}(\tilde{L}_1))\subseteq L_0\cap p(\tilde{\varphi}(\tilde{L}_1))=L_0\cap \varphi(p(\tilde{L}_1))=L_0\cap \varphi(L_1)= \varnothing$, so $\tilde{L}_0$ and $\tilde{L}_1$ are displaceable, as wanted.
\end{proof}

% In fact, $\tilde{\varphi}$ is a Hamiltonian diffeomorphism: let $H_t:X\to \R$ be the time-dependent Hamiltonian that defines the flow $\varphi_t$ of the Hamiltonian vector field $X_{H_t}$ satisfying $\iota_{X_{H_t}}\omega = dH_t$. We can define the unique vector field $\widetilde{X}_{H_t}:\tilde{X}\to T\tilde{X}$ which is a lift of $X_{H_t}$ by $\widetilde{X}_{H_t}(\tilde{x})=(dp_{\tilde{x}})^{-1}(X_{H_t}(p(\tilde{x})))$. Now note that 

   % $$\iota_{\widetilde{X}_{H_t}}p^*\omega = p^*(\iota_{X_{H_t}}\omega)=p^*(dH_t)=d(H_t\circ p)$$
    
  %  Therefore $\widetilde{X}_{H_t}=X_{H_t\circ p}$ is precisely the Hamiltonian vector field for $H_t\circ p$. So if the flow of $X_{H_t\circ p}$ is $\psi$, we must have that $p\circ \psi = {\varphi} \circ p$, so $\psi$ is a lift of $\varphi\circ p$. By uniqueness of lifts, $\psi=\tilde{\varphi}$.

%\section{Consequences and Future Work}

%The proof of Theorem \ref{thm: area inequality} resolves Conjecture 5.6 in \cite{Tao integration approach}, and therefore resolves all the other conjectures in that paper (except for the original square peg conjecture). These include Conjecture 4.1 (proved in \cite{Hugelmeyer}),  Conjecture 5.2, which Tao calls the \textit{quadripartite} square peg problem, and the equivalent Conjectures 6.1, 6.2 and 6.8. 

%In particular, the quadripartite square peg problem asks if the quadruplet $(\sigma_1(\R/\Z),\sigma_1(\R/\Z),\sigma_2(\R/\Z),\sigma_2(\R/\Z))$ jointly inscribes a square where $\sigma_1,\sigma_2:\R/\Z\to Cyl$ are simple closed curves homologous to $\R/\Z$ (note that they are not required to be polygonal).

\end{document}